\documentclass[11pt, final]{article}
\usepackage[utf8]{inputenc}
\usepackage[english]{babel}
\usepackage{csquotes}
\usepackage{comment}
\usepackage{amsmath}
\usepackage{amsthm}
\usepackage{amssymb}
\usepackage{commath}
\usepackage{derivative}
\usepackage{mathrsfs}
\usepackage{color}
\usepackage{tensor}
\usepackage{breqn}
\usepackage{hyperref}
\usepackage[a4paper, total={6in, 9in}]{geometry}

\newcommand{\supnorm}[1]{\left\lVert #1 \right\rVert_{\infty}}

\newcommand{\pnorm}[2]{\left\lVert #2 \right\rVert_{#1}}

\newtheorem{theorem}{Theorem}[section]
\newtheorem{corollary}{Corollary}[theorem]
\newtheorem{lemma}[theorem]{Lemma}

\title{A Non-Conservative, Non-Local Approximation of the Burgers Equation}
\author{
	Shyam Sundar Ghoshal\thanks{Centre for Applicable Mathematics, Tata Institute of Fundamental Research, e-mail: ghoshal@tifrbng.res.in, venkatesh2020@tifrbng.res.in} \and
	Parasuram Venkatesh\footnotemark[1] \and
	Emil Wiedemann\thanks{Department of Mathematics, Friedrich-Alexander-Universit\"at Erlangen-N\"urnberg, e-mail: emil.wiedemann@fau.de}
}
\date{}

\begin{document}
	\maketitle
	\begin{abstract}
		The analysis of non-local regularisations of scalar conservation laws is an active research program. Applications of such equations are found in the modelling of physical phenomena such as traffic flow. In this paper, we propose an inviscid, non-local regularisation in non-divergence form. The salient feature of our approach is that we can obtain sharp a priori estimates on the total variation and supremum norm, and justify the singular limit for Lipschitz initial data up to the time of catastrophe. For generic conservation laws, this result is sharp, since we can demonstrate \textit{non}-convergence when the initial data features simple discontinuities. Conservation laws with linear flux derivative, such as the Burgers equation, behave better even in the presence of discontinuities. Hence, we devote special attention to the limiting behaviour of non-local solutions with respect to the Burgers equation for a simple class of initial data.
		
		\smallskip
		
		
		\smallskip
		
	\end{abstract}
	
	\tableofcontents
	
	\section{Introduction}
	The Burgers equation is a canonical instance of a scalar conservation law. It is often used as a simple example to motivate the study of non-linear hyperbolic PDEs, since it exhibits many of the characteristic features of such equations. In conservative form, it is cast as a Cauchy problem on the domain $[0,\infty)\times\mathbb{R}$ as follows:
	\begin{equation}\label{burger}
		\begin{split}
			\partial_tu+\partial_x\left(\dfrac{u^2}{2}\right)&=0, \\
			u(0,x)&=u_0(x),
		\end{split}
	\end{equation}
	for $u_0\in L^{\infty}(\mathbb{R})$. Generic solutions, even for smooth initial data, break down in finite time, and therefore motivate the notion of \textit{weak} solutions to \eqref{burger}, defined as $u\in C([0,\infty);L^1_{loc}(\mathbb{R}))$ such that, for all $\varphi\in C_c^{\infty}(\mathbb{R}\times[0,\infty))$:
	\[
	\int_{0}^{\infty}\int_{-\infty}^{\infty}u(t,x)\partial_t\varphi(t,x)+\dfrac{u(t,x)^2}{2}\partial_x\varphi(t,x)dxdt+\int_{-\infty}^{\infty}u_0(x)\varphi(0,x)dx=0.
	\]
	However, weak solutions to such initial value problems are not unique. A further condition is imposed on the class of weak solutions to any particular Cauchy problem to single out a unique solution. In the context of scalar conservation laws, the admissibility criterion is that the weak solution must satisfy a further inequality of the form
	\[
	\partial_tn(u)+\partial_xq(u)\leq0,
	\]
	in the sense of distributions, where $n(u)$ is a $C^1$, convex function of its argument, and $q(u)$ is an antiderivative of $n^{\prime}(u)u$ (for the Burgers equation). Such a pair of functions $n,q$ is called an entropy-entropy flux pair with respect to the PDE. By taking $n(u)=\pm u$, we can see that every entropy solution is necessarily an admissible weak one as well. Another equivalent criterion is to insist that they arise as the `vanishing viscosity' limits of certain parabolic equations \cite{kruzkov}. In particular, consider the family of solutions, parametrised by $\epsilon>0$, of the parabolic equation(s)
	\begin{equation}\label{burgers parabolic}
		\begin{split}
			\partial_tu^{\epsilon}+\partial_x\left(\dfrac{(u^{\epsilon})^2}{2}\right)&=\epsilon\Delta u^{\epsilon}, \\
			u(0,x)&=u_0(x).
		\end{split}
	\end{equation}
	For each $\epsilon>0$, solutions to \eqref{burgers parabolic} are obtained by a fixed-point argument. The criterion now is that $u$ must be a limit point of the sequence $u^{\epsilon}$ as $\epsilon\to0^{+}$. Using a priori bounds on the total variation and uniform Lipschitz continuity in time with respect to the $L^1$ norm, the sequence $u^{\epsilon}$ is shown to be compact, and it is demonstrated that the limit is an entropy solution $u$ of the Cauchy problem \eqref{burger} with the given initial data \cite{kruzkov}. Uniqueness of the limit, and therefore justification of the entropy/vanishing viscosity criterion, is proved in a separate argument.
	
	This method uses a second-order parabolic PDE, known the viscous Burgers equation due to the Laplacian. Since this second-order term has a coefficient of $\epsilon$ that vanishes in the limit, this approach is known as the method of `vanishing viscosity'. In this paper, we propose an alternative first-order approximation schema for the Burgers equation using non-linear transport equations. Our motivation is as follows: For smooth and uniformly Lipschitz initial data, a local-in-time solution for \eqref{burger} can be obtained by the method of characteristics. Applying the chain rule to the equation, we see that Lipschitz solutions satisfy
	\[
	\partial_tu+u\partial_xu=0.
	\]
	Thus, the solution is constant along characteristics, which in this case are straight line in spacetime with slope determined by the value of $u_0$ at the base point (at $t=0$). This can only be done locally, because the characteristic curves for generic solutions will meet in finite time. Avoiding the formation of such discontinuities, or `shocks', is the motivation for our approximation schema.
	
	There are two obvious ways to regularise the Burgers equation by convolution, depending on whether one uses the conservative or the transport formulation. Thus, the non-linear term could be replaced either by $\frac12\partial_x([\eta_{\epsilon}\ast u]u)$ or by $(\eta_{\epsilon}\ast u)\partial_{x}u$. We choose the latter, indeed for fixed $\epsilon>0$, we will analyse the following non-conservative, non-local Cauchy problem:
	\begin{equation}\label{nonconservative nonlocal}\tag{NN}
		\begin{split}
			\partial_tu^{\epsilon}+(\eta_{\epsilon}\ast u^{\epsilon})\partial_{x}u^{\epsilon}&=0; \\
			u^{\epsilon}(0,x)&=u_0(x).
		\end{split}
	\end{equation}
	Here, and in the rest of this paper, ${\eta_{\epsilon}}$ is a symmetric family of mollifiers, parametrised by $\epsilon>0$, with compact support $[-\epsilon,\epsilon]$. For instance, consider a smooth, symmetric, positive function $\eta(x)$ supported in $[-1,1]$ such that $\pnorm{1}{\eta}=1$, and for any $\epsilon>0$ let $\eta_{\epsilon}=\epsilon^{-1}\eta(\epsilon^{-1}x)$. For example, consider
	\[
	\eta(x)=I^{-1}\exp\left(\dfrac{-1}{1-x^2}\right)\chi_{(-1,1)}(x),
	\]
	where $\chi_{(-1,1)}(x)$ is the indicator function of the interval $(-1.1)$ and $I$ is a constant that ensures normalisation. We will consider such families $\{\eta_{\epsilon}\}$, though the particular form of $\eta$ is not important as long as it is symmetric with compact support in $[-1,1]$. In some cases, the symmetry condition can also be relaxed. In the singular limit as $\epsilon\to0^+$, convolution with members of this family approximates identity.
	
	Although \eqref{nonconservative nonlocal} is specifically a regularisation of the Burgers equation, we will also briefly consider appropriate non-conservative, non-local regularisations of general scalar conservation laws in section \ref{smooth regime}. As we shall see, however, the Burgers case is special with respect to discontinuous initial data, and therefore singled out for special attention. It is also simpler to work with \eqref{nonconservative nonlocal}, and since the well-posedness results trivially carry over, mutatis mutandis, to the appropriate non-conservative, non-local regularisation of any other scalar conservation law, we can work with this equation without any loss of generality.
	
	We exploit the transport equation structure of \eqref{nonconservative nonlocal} to prove well-posedness by a fixed-point argument and obtain further a priori estimates on solutions. After demonstrating well-posedness for $C^1$ data by a fixed-point argument, we introduce the notion of weak solutions for \eqref{nonconservative nonlocal} and generalise the well-posedness result to $BV$ data, using the a priori estimates derived in the previous step. One advantage of our non-conservative regularisation is that we can prove that the $L^{\infty}$ norm and $BV$ seminorm are preserved for positive times, which greatly simplifies our analysis.
	
	Since the paper of Zumbrun \cite{Zum}, there has been a steady output of results on non-local conservation and balance laws, where spatial mollification is used in place of small viscosity to gain and/or preserve regularity. In all nonlocal approximations of the local models, the main object of interest is to find out when the formal singular limit to the Dirac delta distribution can justified rigorously. Investigations, both numerical and theoretical, have been saliently focused on models of traffic flow \cite{multilane,lwrnonlocal,lwrnonlocalwp}. Other applications include sedimentation \cite{sed}, opinion-formation \cite{opform}, supply chains \cite{supply}, and more. On the theoretical side existence, uniqueness \cite{survey,bvkernel,Crippa2013}, and control problems \cite{boundary,control} have been studied.
	
	The regularisation \eqref{nonconservative nonlocal} was first studied by Norgard and Mohseni, who dubbed it the ``convective filtered Burgers equation" \cite{Norgard_2008}. In another paper, they derived singular limit results for this equation for a class of continuously differentiable initial data \cite{ConFil}. The regularisation of the Burgers equation in non-conservative form was also independently studied by Coron et al. \cite{NLtransport}, along with analogous regularisations of general scalar conservation laws. Our approach is closer to that of the latter, but differs slightly from both; we obtain the singular limit for a larger class of initial data and also generalise the framework to the multidimensional case. Non-local transport equations were also studied in higher dimensions by De Lellis et al. \cite{NLsource}, but in their case the non-locality was introduced through the source term, whereas our approach regularises the velocity field and introduces non-linearity by making the field depend on the solution itself.
	
	In general, the differential equation remains in divergence form even at the non-local level, and the flux function is applied to the regularisation. In general, the non-locality does not behave well in the presence of shocks; one feature of our work is that we can work explicitly with discontinuous initial data, and justify the formal convergence to the entropy solution in the context of Burgers equation for a special class of initial data. Our choice of framework also allows us to avoid a set of counter-examples to convergence (see section \ref{conclusion}), at least locally in time, that were obtained by Colombo et. al. \cite{counters}, where they worked with the nonlocal equation
	\begin{equation}\label{conservative nonlocal}
		\partial_tu+\partial_x((\eta_{\epsilon}\ast u)u))=0,
	\end{equation}
	and proved non-convergence for some initial values. In contrast, we can justify the formal limit as $\epsilon\to0^+$ for any Lipschitz initial data, at least locally in time.
	
	Further results were obtained in the framework of a completely asymmetric, `anisotropic' convolution kernel in \cite{anisotropic}, where convergence to the local limit was obtained under the assumption of a one-sided Lipschitz condition on the local limit. In this paper, however, we will restrict ourselves to symmetric convolution kernels approximating identity.
	
	\subsection{Main results}
	Let us colloquially summarise our main results:
	\begin{itemize}
		\item For smooth data and fixed regularisation parameter $\epsilon>0$, the non-conservative nonlocal Burgers equation has a unique global smooth solution (Theorem~\ref{classical cauchy});
		\item For $BV$ initial data, there exists a unique global weak solution, and no entropy condition is required for uniqueness (Theorem~\ref{weaksol} and Lemma~\ref{L1 stability});
		\item As long as the entropy solution to the local conservation law remains Lipschitz, the non-conservative, non-local approximation(s) with the same initial datum converge uniformly, as $\epsilon\to0$, to the former (Theorem~\ref{smoothconv});
		\item For the Riemann problem, there are examples of convergence but also of non-convergence of the non-local approximations to the local entropy solution (Lemma~\ref{Riemann Problem}).
	\end{itemize}
	Convergence in the Lipschitz regime (Theorem~\ref{smoothconv}) most clearly brings out the improvement of our regularisation schema, where we can justify the formal singular limit with a simple proof, with much weaker assumptions on the regularity of the local limit. By contrast, the proof in Zumbrun \cite{Zum} requires the local solution to be $C^4$ for the local-in-time convergence result. For the non-conservative regularisation in particular, this is also an improvement over the convergence obtained for quasi-concave initial data in \cite{NLtransport}, albeit only locally in time. Thus, in the Lipschitz regime where the weak solutions satisfies the partial differential equation point-wise almost everywhere, our results are sharp. We also do not require the special form of Burgers equation for this result.
	
	Another feature that stands out in our framework is the impossibility of total variation blow-up, as in \cite{tvblowup}. Given the transport structure of our equation, the non-local solutions preserve the total variation of the initial data for all positive times (Lemma~\ref{TVP}).
	
	Generating total variation bounds from a transport structure is not new. For the conservative regularisation schema, this method was used to obtain bounds for the non-local term, which then translated to compact embedding of the solution sequence itself under certain assumptions on the kernel \cite{ccdnkp}, in particular it is highly non-symmetric. One of the questions posed in the paper is precisely the convergence or lack thereof in the presence of a symmetric kernel. As shown in \cite{monodatanonlocal}, conservative non-local regularisation with a symmetric kernel cannot in general satisfy the maximum principle, and total variation may blow up as in \cite{tvblowup}, but in our schema both these possibilities are ruled out. Thus, if convergence must fail, it must fail in some \textit{other} way.
	
	For monotone initial data, we nearly recover the convergence results established in the conservative regularisation framework \cite{gottlich,monodatanonlocal}. If the data is non-decreasing \textit{and Lipschitz}, or non-increasing \textit{and piece-wise constant}, then we can justify the formal singular limit in arbitrary, compact intervals of time. Of course, this is only true with respect to regularisations of the Burgers equation or its rescaled variants; trivial counterexamples to convergence in the case of piece-wise constant, monotone decreasing data can be derived from Lemma~(\ref{general riemann}) otherwise. 
	
	\subsection{Structure of the paper}
	The structure of this paper is as follows: in subsection \ref{c1wp} we demonstrate well-posedness of the Cauchy problem for $C^1$ initial data with fixed $\epsilon>0$, and derive a priori estimates on the $L^{\infty}$ norm, along with the Lipschitz and Total Variation semi-norms. Then, in subsection \ref{weaksol}, we define weak solutions to \eqref{nonconservative nonlocal} and obtain well-posedness for less regular BV initial data using the a priori estimates derived before. We present the multidimensional non-conservative regularisation for scalar conservation laws with general flux in subsection \ref{multidim}. In section \eqref{singlimit}, we justify passage to the singular limit for some classes of initial data, and also provide examples of data for which convergence to the entropy solution does not take place. Section \ref{isen} discusses a simple extension to the case of the isentropic Euler system with cubic pressure law, which is known to have some nice properties that simplify the analysis. Finally, section \ref{conclusion} concludes.
	
	\section{Well-posedness of the Cauchy problem}
	\subsection{Smooth regime}\label{c1wp}
	In this section we analyse the Cauchy problem for $C^1$ initial data and derive some a priori estimates on solutions. For $T>0$, let $\Omega_T=[0,T]\times\mathbb{R}$; we will use this notation throughout the paper to denote such domains. By $\operatorname{Lip}(\mathbb R)$ we denote the space of all functions in $L^\infty(\mathbb R)$ that are globally Lipschitz continuous; in other words, $\operatorname{Lip}(\mathbb R)=W^{1,\infty}(\mathbb R)$.
	\begin{theorem}\label{classical cauchy}
		For $u_0\in C^1(\mathbb{R})\cap\operatorname{Lip}(\mathbb{R})$, the Cauchy problem \eqref{nonconservative nonlocal} has a unique $C^1$ solution in the domain $\Omega_T$ for arbitrary $T>0$.
	\end{theorem}
	\begin{proof}
		Let $g\in L^{\infty}(\Omega_T),u_0\in C^1(\mathbb{R})\cap\operatorname{Lip}(\mathbb{R})\cap L^{\infty}(\mathbb{R})$, and consider the Cauchy problem for the following linear transport equation:
		\begin{equation}\label{linear transport}
			\begin{split}
				\partial_tu+(\eta_{\epsilon}\ast g)\partial_{x}u&=0, \\
				u(0,x)&=u_0(x).
			\end{split}
		\end{equation}
		This Cauchy problem has a unique solution $u$. Define an operator $\Lambda$ that maps $g\in L^{\infty}([0,T]\times\mathbb{R})$ to the unique $u$ solving \eqref{linear transport}. For small enough time $T^*$, we can show that $\Lambda$ is a contraction with respect to the $L^{\infty}$ norm. Then by the Banach fixed point theorem, we can prove the existence of a solution to the Cauchy problem \eqref{nonconservative nonlocal} with data $u_0\in C^1(\mathbb{R})\cap\operatorname{Lip}(\mathbb{R})$.
		
		Given any $g\in L^{\infty}(\Omega_T)$, \eqref{linear transport} becomes a linear hyperbolic initial value problem that can be solved by the method of characteristics. Furthermore, the spatial regularity of the initial data is preserved at positive times. Since $g\in L^{\infty}$, $(\eta_{\epsilon}\ast g)$ is uniformly Lipschitz in $x$ at all non-negative times, and hence the characteristics, i.e the integral curves of $(\eta_{\epsilon}\ast g)$, are well-defined. Thus $\Lambda g=u$ at any given time level is some transportation of the initial data along suitable integral curves. In particular, let $(t,x)\in\Omega_T$. Let $y_{t,x}:[0,t]\to\mathbb{R}$ denote the solution of
		\begin{equation}\label{characteristic ODE}
			\begin{split}
				\dot{y}_{t,x}(s)&=[\eta_{\epsilon}\ast g](s,y_{t,x}(s)), \\
				y_{t,x}(t)&=x.
			\end{split}
		\end{equation}
		Since $u(t,\cdot)$ is a transportation of the initial data along characteristic curves, we have that $u(t,x)=u_0(y_{t,x}(0))$, thus justifying our assertion that the spatial regularity of $u_0$ is preserved for positive times. More precisely, $\|\partial_x u\|_{L^\infty(\Omega_T)}$ is finite and depends only on $\|u_0\|_{C^1}$ and $\|g\|_{L^\infty(\Omega_T)}$.
		
		Additionally, we have a maximum principle for $\eqref{linear transport}$, which we will show also hold for \eqref{nonconservative nonlocal} later. This maximum principle allows us to restrict the operator $\Lambda$, viewing it as a map $\overline{B_R(0)}\to \overline{B_R(0)}$ for $R:=\|u_0\|_{L^\infty}$, where the ball is measured in the norm of $L^\infty(\Omega_T)$.
		
		Consider $g,h\in L^{\infty}(\Omega_T)$ and let $u=\Lambda g,v=\Lambda h$. Transport equations preserve regularity of the initial data, so $u,v\in C^1(\Omega_T)$. If we define $w:=u-v$, then $w$ solves the differential equation
		\[
		\partial_tw+(\eta_{\epsilon}\ast g)\partial_{x} w+(\eta_{\epsilon}\ast(g-h))\partial_{x}v=0.
		\]
		Hence, we have the a priori estimate
		\begin{equation}\label{a priori sup bound}
			\supnorm{w(T^*,\cdot)}\leq\supnorm{g-h}\supnorm{\partial_{x}v}T^*.
		\end{equation}
		As noted before, $\supnorm{\partial_{x}v}$ depends only on $\|h\|_\infty\leq R$ and $\|u_0\|_{C^1}$, both of which are fixed from the initial datum. Given $v$ and for small enough $T^*>0$, \eqref{a priori sup bound} proves that $\Lambda$ is a contraction in $L^{\infty}(\Omega_{T^*})$, so we can deduce the existence of a unique fixed point of $\Lambda$ in $L^{\infty}(\Omega_{T^*})$, which exactly corresponds to a local-in-time solution of the Cauchy problem \eqref{nonconservative nonlocal} with initial data $u_0$.
		
		Now we only need an a priori bound on the Lipschitz constant of the solution at positive times in order to turn this into a global existence result, drawing on the well-known `continuous induction' argument as elaborated in \cite{Hormander}. For a fixed $\epsilon>0$ and $g\in L^{\infty}(\mathbb{R})$, $(\eta_{\epsilon}\ast g)$ is Lipschitz, and in particular by the mean value theorem,
		\[
		\abs{\eta_{\epsilon}\ast g(x_1)-\eta_{\epsilon}\ast g(x_2)}\leq\int_{-\infty}^{\infty}\abs{g(y)}\abs{\eta_{\epsilon}(x_1-y)-\eta_{\epsilon}(x_2-y)}dy\leq\supnorm{g}C_{\epsilon}\abs{x_1-x_2}
		\]
		for any $x_1,x_2\in\mathbb{R}$, where $C_{\epsilon}$ is the Lipschitz constant for $\eta_{\epsilon}$. Solutions are transported along characteristics, and by the maximum principle $\eta_{\epsilon}\ast u$ is uniformly Lipschitz in $x$ for all non-negative times, so $Y_{T}:\mathbb{R}\to\mathbb{R}$ defined as $Y_T(x):=y_{T,x}(0)$ via \eqref{characteristic ODE} is Lipschitz with constant $\exp({\supnorm{u_0}C_{\epsilon}T})$. Hence,
		\[
		\abs{u(T,x_1)-u(T,x_2)}=\abs{u_0(Y_{T}(x_1))-u_0(Y_T(x_2))}\leq\operatorname{Lip}(u_0)\exp(\supnorm{u_0}C_{\epsilon}T)\abs{x_1-x_2},
		\]
		which gives us the required a priori bound on the spatial Lipschitz constant of the solution for positive times. Hence, the spatial derivative of $u$ cannot blow up in finite time, and therefore we have a unique $C^1$ solution to the initial value problem \eqref{nonconservative nonlocal} in $\Omega_T$ for $C^1$, Lipschitz initial data. $T>0$ is arbitrary, so we are done.
	\end{proof}
	The following lemma will be useful later. Throughout this paper, we will slightly abuse notation and denote the total variation by $\pnorm{TV}{\cdot}$, despite the fact that it is only a semi-norm.
	\begin{lemma}\label{TVP}
		Let $u$ be a solution to the Cauchy problem \eqref{nonconservative nonlocal} with $u_0\in BV(\mathbb{R})\cap C^{1}(\mathbb{R})\cap\operatorname{Lip}(\mathbb{R})$. Then, for all $t\geq0:u(t,\cdot)\in BV(\mathbb{R})$ and $\pnorm{TV}{u(t,\cdot)}=\pnorm{TV}{u_0}$, i.e., we have a Total Variation Preservation property.
	\end{lemma}
	\begin{proof}
		We exploit the fact that $u$ is a fixed point of the mapping $\Lambda$ as defined in Theorem \ref{classical cauchy} and therefore the solution to a transport equation. Let the characteristics of $u$ be denoted by the curves $y_{t,x}$ in \eqref{characteristic ODE}. Let $T>0$, and consider any increasing sequence of real numbers $x_1,\ldots,x_n$. Since $u$ is preserved along its own characteristics, we have that
		\[
		\sum_{i=1}^{n-1}\abs{u(t,x_{i+1})-u(t,x_i)}=\sum_{i=1}^{n-1}\abs{u_0(y_{t,x_{i+1}}(0))-u_0(y_{t,x_i}(0))}\leq\pnorm{TV}{u_0}.
		\]
		Taking supremum over all such partitions yields
		\[
		\pnorm{TV}{u(t,\cdot)}\leq\pnorm{TV}{u_0}.
		\]
		To prove the reverse inequality, choose $\delta>0$ and let $x_{1},\ldots,x_m$ be a strictly increasing sequence of real numbers such that 
		\[
		\pnorm{TV}{u_0}-\delta<\sum\abs{u_0(x_{i+1})-u_0(x_i)}.
		\]
		The existence of such a partition is guaranteed by the definition of total variation. Let $y_{i}$ denote the integral curves of $\eta_{\epsilon}\ast u$ that solve the corresponding \textit{initial} value problem(s) $y_i(0)=x_i$. Then for $t>0$, the sequence $y_1(t),\ldots,y_m(t)$ is a strictly increasing sequence of real numbers, and therefore
		\[
		\pnorm{TV}{u(t,\cdot)}\geq\sum_{i=1}^{m-1}\abs{u(t,y_{i+1}(t))-u(t,y_i(t))}=\sum_{i=1}^{m-1}\abs{u_0(x_{i+1})-u_0(x_i)}>\pnorm{TV}{u_0}-\delta.
		\]
		Since $\delta>0$ can be arbitrarily chosen, we are done.
	\end{proof}
	The solution $u$ can also be shown uniformly Lipschitz continuous in time over compact intervals.
	\begin{lemma}\label{ep Lip}
		Solutions of the Cauchy problem \eqref{nonconservative nonlocal} with Lipschitz initial data, for all $T>0$. are also uniformly Lipschitz on the domains $\Omega_T$.
	\end{lemma}
	\begin{proof}
		It is enough to prove Lipschitz continuity with respect to time, at the initial time in terms of the spatial Lipschitz constant and $L^{\infty}$ norm. Let $h>0$, then by the mean value theorem we have that, for some $\lambda\in(0,1)$:
		\[
		\begin{split}
			\abs{u(h,x)-u_0(x)}&=h\abs{\partial_tu(\lambda h,x)}h \\
			&=\abs{(\eta_{\epsilon}\ast u(\lambda h,x))\partial_xu(\lambda h,x)}h \\
			&\leq\supnorm{\partial_xu(h,\cdot)}\supnorm{u_0}h.
		\end{split}
		\]
		Hence,
		\[
		\dfrac{\abs{u(h,x)-u_0(x)}}{h}\leq\supnorm{\partial_xu(h,\cdot)}\supnorm{u_0}.
		\]
		This concludes the proof.
	\end{proof}	
	Since we have a priori control on $\supnorm{\partial_xu(t,\cdot)}$ for all positive times by \eqref{a priori sup bound}, we can conclude that $u$ is uniformly Lipschitz in compact intervals of time. Using this, we prove Lipschitz continuity with respect to the  $L^1$ norm. Importantly, this is established independent of $\epsilon$.
	\begin{lemma}\label{Lipschitz L1}
		For any fixed $T>0$, the solution $u$ of the initial value problem \eqref{nonconservative nonlocal} with $C^1, BV$ data $u_0$ is uniformly Lipschitz in time with respect to the $L^1$ norm, i.e, there exists $K$ such that for all $t_1,t_2\leq T$:
		\[
		\pnorm{1}{u(t_2,\cdot)-u(t_1,\cdot)}\leq K(t_2-t_1).
		\]
	\end{lemma}
	\begin{proof}
		By Lemma~\ref{ep Lip}, $u$ is Lipschitz in time, so by the fundamental theorem of calculus
		\[
		u(t_2,x)-u(t_1,x)=\int_{t_1}^{t_2}\partial_tu(s,x)ds.
		\]
		From \eqref{nonconservative nonlocal} and Lemma~\ref{TVP}, we can bound the $L^1$ norm of this difference as follows:
		\[
		\begin{split}
			\int_{\mathbb{R}}\abs{u(t_2,x)-u(t_1,x)}dx&=\int_{\mathbb{R}}\abs{\int_{t_1}^{t_2}\partial_tu(s,x)ds}dx \\
			&\leq\int_{\mathbb{R}}\int_{t_1}^{t_2}\abs{(\eta_{\epsilon}\ast u)\partial_{x}u}dx \\
			&=\int_{t_1}^{t_2}\int_{\mathbb{R}}\abs{(\eta_{\epsilon}\ast u)\partial_{x}u}dx \\
			&\leq\int_{t_1}^{t_2}\supnorm{u(s,\cdot)}\pnorm{1}{\partial_{x}u(s,\cdot)}ds \\
			&\leq\supnorm{u_0}\pnorm{TV}{u_0}(t_2-t_1),
		\end{split}
		\]
		which completes the proof, with $K=\supnorm{u_0}\pnorm{TV}{u_0}$.
	\end{proof}
	Thus, a uniform Lipschitz property, and therefore equicontinuity in time with respect to the $L^1$ norm is established independent of $\epsilon$. This will help us in the next section to establish the well-posedness of \eqref{nonconservative nonlocal} for less regular initial data. In higher dimensions however, the total variation estimate no longer holds independently of $\epsilon$, and hence we cannot justify the passage to the limit, although the solution remains $BV$ for $BV$ initial data.
	
	\subsection{Weak solutions}\label{weaksol}
	Smooth solutions to the nonlocal equation exist globally in time for smooth initial data. However, we also want to work with less regular initial data, and unlike the case of parabolic approximations, our solutions do not gain any regularity. Thus, we need to define admissible weak solutions. We say that $u\in C([0,\infty);L^{1}_{loc}(\mathbb{R}))$ is a weak solution to the Cauchy problem \eqref{nonconservative nonlocal} with initial data $u_0$ if, for all $\varphi\in C_c^{\infty}([0,\infty)\times\mathbb{R})$:
	\[
	\iint_{\mathbb{R}^{2}_{+}}u\left[\partial_t\varphi+\partial_{x}(\varphi(\eta_{\epsilon}\ast u))\right]dxdt+\int_{\mathbb{R}}u_0(x)\varphi(0,x)dx=0.
	\]
	Since $\eta_{\epsilon}\ast u$ is always smooth regardless of the regularity of $u$, this integral is always well-defined, and every classical solution of the initial value problem is also a weak one. From the a priori estimates on the total variation, and the uniform Lipschitz time-regularity with respect to the $L^1$ norm, we can now deduce the existence of weak solutions for $BV$ initial data by a diagonal argument.
	\begin{theorem}
		The Cauchy problem \eqref{nonconservative nonlocal} has a weak solution $u\in  C([0,T];L^{1}_{loc}(\mathbb{R}))$ for all $T>0$, such that $\supnorm{u}\leq\supnorm{u_0}$, and also for all $t_1,t_2\geq0:$
		\[
		\pnorm{L^1}{u(t_1,\cdot)-u(t_2,\cdot)}\leq\pnorm{BV}{u_0}\supnorm{u_0}\abs{t_1-t_2}.
		\]
	\end{theorem}
	\begin{proof}
		Given $u_0\in BV(\mathbb{R})$, let $u_0^{n}(x)=\eta_{1/n}\ast u_0(x)$. Then for all $n,u_0^n\in BV(\mathbb{R})\cap C^1(\mathbb{R})\cap\operatorname{Lip}(\mathbb{R})$. By Theorem \ref{classical cauchy}, the Cauchy problem \eqref{nonconservative nonlocal} with initial data $u_0^{n}$ has a classical solution for each $n\in\mathbb{N}$, say $u^n$, defined globally. Furthermore, $\supnorm{u_0^{n}}\leq\supnorm{u_0}$ and $\pnorm{TV}{u_0^{n}}\leq\pnorm{TV}{u_0}$, so we have that for all $t>0:$
		\[
		\pnorm{TV}{u^n(t,\cdot)}=\pnorm{TV}{u_0^n}\leq\pnorm{TV}{u_0},
		\]
		and for all $t_1,t_2>0:$
		\[
		\pnorm{L^1}{u^{n}(t_1,\cdot)-u^n(t_2,\cdot)}\leq\supnorm{u_0^n}\pnorm{TV}{u_0^n}\abs{t_1-t_2}\leq\supnorm{u_0}\pnorm{TV}{u_0}\abs{t_1-t_2}.
		\]
		The sequence $\{u^n\}$ is comprised of weak solutions with initial data converging to $u_0$. By Helly's theorem, $\{u^n(t,\cdot)\}$ is uniformly bounded in BV for all $t>0$ and therefore compact in $L^1([-K,K])$ for any $K>0$.
		
		Thus, by a diagonalization argument along the lines of \cite{timecompactness}, we can extract a subsequence $u^{n_k}$ that converges to some $u\in L^1_{loc}([0,\infty)\times\mathbb{R})$ pointwise almost everywhere and in $L^1$ on compact sets. The limit $u$ also inherits the Lipschitz time-regularity with respect to the $L^1$ norm. Hence, for any $\varphi\in C_c^{\infty}(\mathbb{R}\times(0,\infty))$ with support (say) $K$:
		\[
		\begin{split}
			&\iint_{K}u[\partial_t\varphi+\partial_x(\varphi(\eta_{\epsilon}\ast u))]dxdt \\
			=&\iint_{K}u[\partial_t\varphi+\partial_x\varphi(\eta_{\epsilon}\ast u)+\varphi(\eta_{\epsilon}^{\prime}\ast u)]dxdt \\
			=&\lim_{n\to\infty}\iint_{K}u^n[\partial_t\varphi+\partial_x\varphi(\eta_{\epsilon}\ast u^n)+\varphi(\eta_{\epsilon}^{\prime}\ast u^n)]dxdt \\
			=&\lim_{n\to\infty}\iint_{K}u^n[\partial_t\varphi+\partial_x(\varphi(\eta_{\epsilon}\ast u^n))]dxdt=0,
		\end{split}
		\]
		by an application of Lebesgue's dominated convergence theorem, since for any smooth function $\kappa$ with compact support and $u^n\to u$ in $L^1$, $\kappa\ast u^n\to\kappa\ast u$ pointwise almost everywhere, and we can simply take the dominating function to be $\supnorm{u_0}$ by the maximum principle because we can restrict ourselves to a compact domain $K$. This completes the proof; by time-regularity with respect to the $L^1$ norm we can see that the initial condition is also satisfied, and thus a weak solution to \eqref{nonconservative nonlocal} exists within the postulated class.
	\end{proof}
	We can now prove the uniqueness of this limit; more precisely we will prove stability with respect to $L^1$ perturbations, from which uniqueness will trivially follow.  As is common in the nonlocal literature, we do not require additional criteria such as entropy conditions to ensure the uniqueness of weak solutions. However, as we shall see, this stability is not independent of $\epsilon$, and thus cannot naïvely help us pass to the limit as $\epsilon\to0$.
	\begin{lemma}\label{L1 stability}
		Let $u,v\in L^{\infty}(\Omega_T)\cap BV(\Omega_T)$ be two weak solutions of \eqref{nonconservative nonlocal} with initial data $u_0,v_0$ respectively such that $(u_0-v_0)\in L^1(\mathbb{R})$. Then, there exists a constant $C_{\epsilon}>0$ depending on $\epsilon$ such that
		\[\pnorm{1}{(u-v)(t,\cdot)}\leq e^{C_{\epsilon}t}\pnorm{1}{u_0-v_0}\]
		Note that if $u_0-v_0\notin L^1(\mathbb{R})$, then the right-hand side is infinite anyway so the inequality would trivially hold.
	\end{lemma}
	\begin{proof}
		Since $u,v$ are of bounded variation, their derivatives are Radon measures. Additionally, $w=\abs{u-v}$ is also of bounded variation, and
		\[
		\begin{split}
			\partial_tw&=\text{sgn}(u-v)\partial_t(u-v) \\
			&=\text{sgn}(v-u)((\eta_{\epsilon}\ast u)\partial_{x}(u-v)+(\eta_{\epsilon}\ast(u-v))\partial_{x}v) \\
			&=-(\eta_{\epsilon}\ast u)\partial_{x}w+\text{sgn}(v-u)(\eta_{\epsilon}\ast(u-v))\partial_{x}v. 
		\end{split}
		\]
		Hence,
		\[
		\partial_tw+\partial_{x}((\eta_{\epsilon}\ast u)w(t,x))=-\text{sgn}(u-v)(\eta_{\epsilon}\ast(u-v))\partial_{x}v+w(\eta_{\epsilon}\ast\partial_{x}u),
		\]
		and so, by the Gauss-Green theorem for $BV$ functions,
		\[
		\odv{}{t}\pnorm{1}{w(t,\cdot)}\leq\pnorm{1}{w(t,\cdot)}\supnorm{\eta_{\epsilon}}\left(\pnorm{TV}{u}+\pnorm{TV}{v}\right),
		\]
		after which the lemma follows from Grönwall's inequality, with
		\[
		C_{\epsilon}=\supnorm{\eta_{\epsilon}}\left(\pnorm{TV}{u}+\pnorm{TV}{v}\right),
		\]
		completing the proof.
	\end{proof}
	We can also prove the existence of weak solutions for general $L^{\infty}$ data by taking a subsequence in the weak-$\ast$ topology of $L^{\infty}$, but uniqueness is not guaranteed as in the BV case. We will not pursue this matter further here.
	
	As a corollary of the uniqueness result, and by analogy with the linear transport equation, we also have finite speed of propagation for the non-local equation, for mass if not information. It is a trivial application of the known results on transport equations as covered in \cite{transport}.
	\begin{corollary}
		Let $u_0\in BV(\mathbb{R})$. Then, letting $y_{t,x}$ denote the integral curves of the (smooth) velocity field $\eta_{\epsilon}\ast u$, we have that $u(t,x)=u_0(y_{t,x}(0))$ almost everywhere, i.e., at all points of continuity.
	\end{corollary}
	Since $u$ obeys a maximum principle, the integral curves have a uniform Lipschitz bound and therefore propagate with finite speed depending only on $\supnorm{u_0}$.
	
	\subsection{The multidimensional problem}\label{multidim}
	Consider the multi-dimensional scalar conservation law
	\begin{equation}\label{loc}
		\partial_tu+\sum_{i=1}^{n}\partial_if_i(u)=0\text{ in }\mathbb{R}^d\times[0,T]
	\end{equation}
	which can also be written in quasilinear form
	\[
	\partial_tu+\sum_{i=1}^{n}f_i^{\prime}(u)\partial_iu=0,
	\]
	for which we can apply a non-local velocity-regularisation schema and obtain
	\begin{equation}\label{nonloc}
		\partial_tu^{\epsilon}+(\eta_{\epsilon}\ast F^{\prime}(u))\cdot\nabla u^{\epsilon}=0,
	\end{equation}
	where
	\[
	F^{\prime}(u)=(f_1^{\prime}(u),\ldots,f_n^{\prime}(u)).
	\]
	We will always assume that $F$ is locally Lipschitz. This is still a non-linear transport equation with convolution as before, and if we define characteristics starting from time zero by $y^{\epsilon}(t;x)$, then they satisfy the differential equation
	\[
	\dot{y}^{\epsilon}(t;x)=[\eta_{\epsilon}\ast(F\circ u)](y^{\epsilon}(t;x),t)\text{ with }y(0,x)=x.
	\]
	By taking one derivative and letting $v^{\epsilon}_i=\partial_iu^{\epsilon}$, we get the transport equation for the derivative as before
	\[
	\partial_tv^{\epsilon}_i+(\eta_{\epsilon}\ast F^{\prime}(u^{\epsilon}))\cdot\nabla v^{\epsilon}_i=-\partial_i(\eta_{\epsilon}\ast F^{\prime}(u^{\epsilon}))\cdot\nabla u^{\epsilon}.
	\]
	Hence, the derivative does not blow up, and the fixed point argument for small time can be extended to a global existence result for the nonlocal equation.
	
	The same method of characteristics/fixed point arguments suffice to conclude well-posedness of the multidimensional non-local equation. The only difference is the total variation estimate, which cannot be obtained in the same way. However, we still have some estimate of the form given below.
	\begin{lemma}
		The total variation of $u^{\epsilon}(\cdot,t)$ at each time level is bounded.
	\end{lemma}
	\begin{proof}
		Unlike the 1D case, we cannot use the supremum-over-partitions definition to easily conclude, and $L^1$ norm of the gradient does not work either because it is $\epsilon$-dependent. Instead we use the dual formulation of the total variation seminorm:
		\[
		\pnorm{TV}{u}=\sup\left\{\int u\operatorname{div}\varphi dx;\varphi\in C_c^{\infty},\supnorm{\varphi}\leq1\right\}.
		\]
		Let $Y_{\epsilon}(x,t):\mathbb{R}^n\times[0,T]\to\mathbb{R}^n$ denote the backward transport function of the nonlocal equation, i.e. $Y_{\epsilon}(x,t)=y^{\epsilon}_x(0)$, where $y^{\epsilon}_{x}(s)$ solves the terminal value problem
		\begin{equation*}
			\begin{split}
				\dot{y}^{\epsilon}_x(s)&=\eta_{\epsilon}\ast u^{\epsilon}(y^{\epsilon}_x(s),s), \\
				y^{\epsilon}_x(t)&=x,
			\end{split}
		\end{equation*}
		in the domain $[0,t]$. Now by the continuous dependence properties of ODEs, $Y^{\epsilon}(\cdot,t)$ is a bilipschitz diffeomorphism of $\mathbb{R}^n$ onto itself. Furthermore, we know that the solution $u^{\epsilon}$ is given as
		\[
		u^{\epsilon}(x,t)=u_0(Y^{\epsilon}(x,t)).
		\]
		Since we are no longer in one spatial dimension, the total variation of $u^{\epsilon}$ is no longer (necessarily) conserved, but bounded by
		\[
		\pnorm{TV}{u^{\epsilon}(\cdot,x)}\leq\supnorm{\partial_xY^{\epsilon}(\cdot,t)}^{d-1}\pnorm{TV}{u_0}.
		\]
	\end{proof}
	Thus, solutions to the nonlocal conservation law remain in BV but we do not have a uniform bound independent of $\epsilon$ to establish pre-compactness. However, we can obtain convergence to the appropriate limit in other ways, as long as the entropy solution is smooth. Analogous results hold for a `flux-regularised' non-local equation of the form
	\begin{equation}\label{nonlocflux}
		\partial_tu^{\epsilon}+F^{\prime}(\eta_{\epsilon}\ast u)\cdot\nabla u^{\epsilon}=0
	\end{equation}
	as well.
	
	For fixed $\epsilon>0$, the solution of the non-local equation(s) remains of bounded variation if the initial data is of bounded variation as well. Hence, the existence and uniqueness of weak solutions, Lipschitz continuity with respect to the $L^1$ norm, and similar properties remain as in the scalar case. However, we can no longer assume a priori that the sequence of solutions as $\epsilon\to0$ is precompact, since the total variation estimate now depends on $\epsilon$.
	
	\section{The non-local to local limit}\label{singlimit}
	We know from the theory of conservation laws that for Lipschitz initial data, the Cauchy problem for the Burgers equation has a Lipschitz solution for small time. In fact, the exact time of blow-up $t^{\ast}$ for the Lipschitz constant can be precisely calculated as
	\[
	t^{\ast}=\left[\mathop{\operatorname{ess\,sup}}_{y\in\mathbb{R}}\left\{-\partial_xu_0(y)\right\}\right]^{-1},
	\]
	with $t^{\ast}=\infty$ if $u_0$ is monotone non-decreasing. In the following subsection, we justify the non-local to local limit in the smooth regime, i.e when the entropy solution of the Burgers equation is known to be Lipschitz. The results we obtain are not specific to Burger equation and work for any scalar conservation law, appropriately regularised, and therefore we will present our results in this slightly more general setting which includes also the multidimensional case.
	
	In their paper, Coron et. al. \cite{NLtransport} demonstrated that every convergent subsequence $u^{\epsilon}$ converges to a \textit{weak} solution of \eqref{burger}, but not necessarily the entropy solution. Passage to the unique entropy solution was also demonstrated for quasi-concave initial data. In this section we expand this convergence result for a broader class of initial data.
	
	\subsection{Smooth regime}\label{smooth regime}
	Let $F\in C^2,F:\mathbb{R}\to\mathbb{R}^n$ be such that $F^{\prime}$ is uniformly Lipschitz with constant $M$, and consider the following scalar conservation law:
	\begin{equation}\label{scl}
		u_t+\nabla_{x}\cdot F(u)=0.
	\end{equation}
	Note that this is not a restrictive assumption, since the Cauchy problems we are working with satisfy a maximum principle with respect to the initial data. When the entropy solution is Lipschitz, we can apply the chain rule to get
	\begin{equation}\label{pwise scl}
		u_t+F^{\prime}(u)\cdot\nabla u=0.
	\end{equation}
	Now, there are two possible regularisations of \eqref{scl} matching our `nonlocal schema' for Burgers equation, as follows:
	\begin{equation}\label{velocity reg}
		u^{\epsilon,1}_t+(\eta_{\epsilon}\ast F^{\prime}(u^{\epsilon,1}))\cdot\nabla u^{\epsilon,1}=0,
	\end{equation}
	\begin{equation}\label{flux reg}
		u^{\epsilon,2}_t+F^{\prime}(\eta_{\epsilon}\ast u^{\epsilon,2})\cdot u^{\epsilon,2}=0,
	\end{equation}
	and the well-posedness theory for the nonlocal Burgers equation illustrated before holds mutatis mutandis. Note that both regularisations are equivalent to \eqref{nonconservative nonlocal} for the Burgers equation, where $f^{\prime}$ is linear. Also note that \eqref{velocity reg} and \eqref{flux reg} are just the cases of \eqref{nonloc} and \eqref{nonlocflux} respectively.
	
	We will show that, locally in time, smooth solutions of \eqref{nonconservative nonlocal} converge to the entropy solution of \eqref{scl} with the same initial data as $\epsilon\to0^+$. The following theorem is independent of spatial dimension, since we only analyse the deviation along characteristics.
	\begin{theorem}\label{smoothconv}
		Consider $u_0\in\operatorname{Lip}(\mathbb{R}^n), F:\mathbb{R}\to\mathbb{R}^n$ such that $F^{\prime}\in\operatorname{Lip}(\mathbb{R})$, and let $\tau>0$ be the maximal time such that the solution $u$ to \eqref{scl} with initial data $u_0$ is uniformly Lipschitz on $\Omega_T$ for all $T<\tau$, with $\Omega_T$ now denoting the domain $[0,T]\times\mathbb{R}^n$. \linebreak
		Let $u$ be the solution to \eqref{burger} with this initial data. Then, the family of solutions to the nonlocal equation(s) \eqref{velocity reg} and \eqref{flux reg} with the same initial data, say $u^{\epsilon,1},u^{\epsilon,2}$, converge uniformly to $u$ as $\epsilon\to0$ on domains $\Omega_T$ when $T<\tau$. On $\Omega_{\tau}$, in the typical case where $\tau$ is finite, we thus have convergence in $L^1$ on compact sets by the dominated convergence theorem.
	\end{theorem}
	\begin{proof}
		For $i=1,2$ define the sequences of functions $w^{\epsilon,i}$ on the domain $\Omega_T$ by
		\[
		w^{\epsilon,i}=u-u^{\epsilon,i}.
		\]
		Let us consider $i=1$; the other case is quite similar. Let $L$ be the uniform Lipschitz constant for $u$ on $\Omega_T$. From \eqref{velocity reg},\eqref{pwise scl}:
		\[
		\begin{split}
			w^{\epsilon,1}_t+(\eta_{\epsilon}\ast F^{\prime}(u^{\epsilon,1}))\cdot\nabla w^{\epsilon,1}&=(\eta_{\epsilon}\ast F^{\prime}(u^{\epsilon,1})-F^{\prime}(u))\cdot\nabla u \\
			&=(\eta_{\epsilon}\ast( F^{\prime}(u^{\epsilon,1})-F^{\prime}(u)))\cdot\nabla u+(\eta_{\epsilon}\ast F^{\prime}(u)-F^{\prime}(u))\cdot\nabla u.
		\end{split}
		\]
		The last equality is simply a transport equation with a source term, due to our regularity assumptions on the entropy solution. Since $f^{\prime}$ is uniformly Lipschitz, for $t<T$ we have
		\[
		\begin{split}
			\supnorm{w^{\epsilon,1}(\cdot,t)}&\leq\supnorm{w^{\epsilon,1}(\cdot,0)}+\int_0^tLM\supnorm{w^{\epsilon,1}(\cdot,s)}ds+\epsilon L^2Mt,
		\end{split}
		\]
		where $M$ is the Lipschitz constant of $f^{\prime}$. Hence, by Grönwall's inequality, we have that on the domain $\Omega_T$:
		\[
		\supnorm{w^{\epsilon,1}}\leq\epsilon L^2MTe^{LMT}.
		\]
		Thus, the nonlocal approximations converge uniformly to the Lipschitz entropy solution as the parameter $\epsilon$ goes to zero. By considering a monotonically increasing sequence of times $T_n$ converging to $\tau$ and uniform convergence on the domains $\Omega_{T_n}$, we also obtain the $L^1_{loc}$ convergence of non-local solutions to the local limit by Lebesgue's dominated convergence theorem. That is to say, for compact spatial domains $K\subset\mathbb{R}^n$, the nonlocal solutions converge to the local solution in $L^1([0,T]\times K)$. Thus, we can justify the formal singular limit up to the time of catastrophe, i.e., gradient blow-up.
	\end{proof}
	
	Once the entropy solution forms discontinuities, however, this argument based on the method of characteristics can no longer be extended. This motivates our analysis of solutions with discontinuous initial data. The simplest such case is the Riemann problem, which we shall turn to next.
	
	\subsection{The Riemann problem}
	We now arrive at our first main result tackling discontinuities, which shows both convergence and non-convergence to the local limit depending on the form of $u_0$. Furthermore, we shall see that the Burgers case really is special in allowing nonlocal-to-local convergence to the entropy solution in the presence of discontinuities.
	\begin{lemma}\label{Riemann Problem}
		Let $u_0(x)$ be a piece-wise constant function taking on exactly two values, i.e., suppose
		\begin{equation}\label{Riemann data}
			u_0(x)=
			\begin{cases}
				&u_L\text{ if } x\leq0, \\
				&u_R\text{ if }x>0.
			\end{cases}
		\end{equation}
		Then the unique BV solution $u^{\epsilon}$ to \eqref{nonconservative nonlocal} with initial data given by \eqref{Riemann data} is independent of $\epsilon$ and given by
		\begin{equation}\label{Riemann solution}
			u(t,x)=
			\begin{cases}
				&u_L\text{ if }x\leq\sigma t, \\
				&u_R\text{ if }x>\sigma t,
			\end{cases}
		\end{equation}
		where $\sigma=\frac{1}{2}(u_L+u_R)$ is the Rankine-Hugoniot shock speed for the Burgers equation. In particular, we have trivial convergence to the entropic weak solution for shock-type Riemann data $(u_L>u_R)$, and non-convergence for rarefaction-type initial data $(u_L<u_R)$.
	\end{lemma}
	\begin{proof}
		If $u_L=u_R$ there is nothing to prove, so we will assume that we are not dealing with the trivial case. As an ansatz, define $u(t,x)$ as in the theorem; we simply have to show that it is a weak solution. Let $\varphi\in C_c^{\infty}([0,\infty)\times\mathbb{R})$. Then, let $\Omega_L=\{(t,x):x\leq\sigma t\},\Omega_R=\{(t,x):x>\sigma t\}$; note that they partition $\Omega=[0,\infty)\times\mathbb{R}$. Also note that, by the symmetry of the mollifier, $\eta_{\epsilon}\ast u(t,\sigma t)=\sigma$. Hence,
		\[
		\begin{split}
			&\int_{0}^{\infty}\int_{-\infty}^{\infty}u[\partial_t\varphi+\partial_x(\varphi(\eta_{\epsilon}\ast u))]dxdt+\int_{-\infty}^{\infty}u_0(x)\varphi(0,x)dx \\
			=&\iint_{\Omega}u[\partial_t\varphi+\partial_x(\varphi(\eta_{\epsilon}\ast u))]dxdt+\int_{-\infty}^{\infty}u_0(x)\varphi(0,x)dx \\
			=&\iint_{\Omega_L}u_L(\partial_t\varphi+\partial_x(\varphi(\eta_{\epsilon}\ast u)))dxdt \\
			&+\iint_{\Omega_R}u_R(\partial_t\varphi+\partial_x(\varphi(\eta_{\epsilon}\ast u)))dxdt \\
			&+\int_{-\infty}^{\infty}u_0(x)\varphi(0,x)dx,
		\end{split}
		\]
		or,
		\[
		\begin{split}
			&\int_{0}^{\infty}\int_{-\infty}^{\infty}u[\partial_t\varphi+\partial_x(\varphi(\eta_{\epsilon}\ast u))]dxdt+\int_{-\infty}^{\infty}u_0(x)\varphi(0,x)dx \\
			=&\dfrac{1}{\sqrt{1+\sigma^2}}\int_{0}^{\infty}u_L[\sigma\varphi(t,\sigma t)-\varphi(t,\sigma t)(\eta_{\epsilon}\ast u(t,\sigma t))]dxdt \\
			&+\dfrac{1}{\sqrt{1+\sigma^2}}\int_{0}^{\infty}u_R[-\sigma\varphi(t,\sigma t)+\varphi(t,\sigma t)(\eta_{\epsilon}\ast u(t,\sigma t))](t,x)dt \\
			=&\dfrac{1}{\sqrt{1+\sigma^2}}\int_{0}^{\infty}(u_L-u_R)\varphi(t,\sigma t)[\sigma-\eta_{\epsilon}\ast u(t,\sigma t)]dt \\
			=&0.
		\end{split}
		\]
		Since $\varphi$ is arbitrary, this shows that the solution to \eqref{nonconservative nonlocal} with initial data \eqref{Riemann data} is indeed \eqref{Riemann solution}, and it is the unique BV solution by Lemma~\ref{L1 stability}.
	\end{proof}
	Note that the particular form of the Burgers flux was crucial to the proof, and it is easy to see that convergence to the entropy solution does not hold in general even for the Riemann problem if we consider regularisations like \eqref{velocity reg} or \eqref{flux reg}. Thus we have the following lemma for the generic Riemann problem with respect to our non-conservative, non-local equation.
	\begin{lemma}\label{general riemann}
		Let $u^{\epsilon,1},u^{\epsilon,2}$ denote solutions to the Cauchy problem(s) \eqref{velocity reg} and \eqref{flux reg} respectively, with initial data
		\[
		u_0(x)=
		\begin{cases}
			&u_L\text{ if } x\leq0, \\
			&u_R\text{ if }x>0.
		\end{cases}
		\]
		Then the solution(s) are independent of $\epsilon>0$ and given by
		\[
		u^{\epsilon,i}(t,x)=
		\begin{cases}
			&u_L\text{ if }x\leq\sigma_i t, \\
			&u_R\text{ if }x>\sigma_i t,
		\end{cases}
		\]
		where
		\[
		\sigma_i=
		\begin{cases}
			&\dfrac{f^{\prime}(u_L)+f^{\prime}(u_R)}{2}\text{ if }i=1, \\
			&f^{\prime}\left(\dfrac{u_L+u_R}{2}\right)\text{ if }i=2.
		\end{cases}
		\]
	\end{lemma}
	The proof is trivially similar to that of Lemma~\ref{Riemann Problem}, and therefore will not be repeated. These two lemmas suggest the following:
	\begin{enumerate}
		\item For fluxes $f$ such that $f^{\prime}$ is non-linear, we cannot expect to justify the singular limit of either \eqref{velocity reg} or \eqref{flux reg} outside domains where the entropy solution itself is smooth.
		\item Even for linear fluxes (i.e. the Burgers equation and its rescaled variants), we cannot expect convergence if the initial data contains a positive jump discontinuity, since the entropy conditions are violated across the jump.
	\end{enumerate}
	
	\subsection{Piece-wise Lipschitz increasing data}
	For $C^1$, Lipschitz initial data, we can differentiate \eqref{nonconservative nonlocal} to obtain the following transport equation for the spatial derivative $\partial_xu^{\epsilon}=v^{\epsilon}$:
	\begin{equation}\label{nonlocal derivative}
		\partial_tv^{\epsilon}+(\eta_{\epsilon}\ast u^{\epsilon})\partial_xv^{\epsilon}=-v^{\epsilon}\partial_x(\eta_{\epsilon}\ast u^{\epsilon}).
	\end{equation}
	Now, if $u^{\epsilon}(x+\delta,0)>u^{\epsilon}(x,0)$ for all $\delta>0$ small enough, then we also have that $u^{\epsilon}(x+\delta^{\prime},t)>u^{\epsilon}(x,t)$ for all $\delta^{\prime}>0$ small enough. That is to say, $v^{\epsilon}$ does not change sign. Hence, along a characteristic $\gamma^{\epsilon}(t)$, we have that $v^{\epsilon}(t,\gamma^{\epsilon}(t))$ is decreasing in time if it is positive, unless $\partial_x(\eta_{\epsilon}\ast u^{\epsilon})(t,\gamma^{\epsilon}(t))=\eta_{\epsilon}\ast v^{\epsilon}(t,\gamma^{\epsilon}(t))<0$.
	
	In particular, for Lipschitz monotone increasing initial data, the Oleinik condition is satisfied and the non-local solutions $u^{\epsilon}$ converge to the entropy solution in $L^1$ on all compact subsets of spacetime. However, we already saw this in Theorem~\ref{smoothconv}. Let us now consider the case of piece-wise Lipschitz increasing data to generalise this idea.
	
	Consider initial data $u_0\in BV(\mathbb{R})$ such that, for $a_1,a_2\in\mathbb{R},a_1<a_2$:
	\[
	u_0(x)=
	\begin{cases}
		&v_0(x),\text{ if }x< a_1, \\
		&v_1(x)\text{ if }x\in[a_1,a_2), \\
		&v_2(x)\text{ if }x\geq a_2,
	\end{cases}
	\]
	where $v_i(x)$ are uniformly Lipschitz increasing functions with constant $C$, and the jump discontinuities satisfy
	\[
	u_0(a_1-)>u_0(a_1+),u_0(a_2-)>u_0(a_2+),
	\]
	with $u_0(x\pm)$ denoting the left and right limits respectively. Note that if the inequalities are not strict, then we are reduced to the Lipschitz increasing case we already covered earlier. In case where one of the inequalities is strict but not the other, we can redefine the other point to be be much further away without loss of generality and for each finite time horizon, perform the following analysis trivially, as we shall see.

	Let $\gamma^{\epsilon}_{j}(t)$ denote the characteristic curves of the Cauchy problem \eqref{nonconservative nonlocal} with initial data $u_0$ as above. Then, from \eqref{nonlocal derivative}, we have that $u^{\epsilon}(\cdot,t)$ is monotone increasing for values in each of the intervals $(-\infty,\gamma^{\epsilon}_1(t)),(\gamma^{\epsilon}_1(t),\gamma^{\epsilon}_2(t)),(\gamma^{\epsilon}_2(t),\infty)$ separately.
	
	Recall that $u^{\epsilon}$ for any sequence of $\epsilon\to0$ is precompact in $L^1_{loc}$, and every convergent subsequence converges to a weak solution. Furthermore, the curves $\gamma^{\epsilon}_i$ are uniformly Lipschitz, hence we can extract subsequences converging uniformly on compact intervals of time. Fix $T>0$, and consider a sequence, still denoted by $\epsilon$, such that $\gamma^{\epsilon}_{i}\to\gamma_{j}$ uniformly on $[0,T]$. We can choose $T$ small enough such that the curves $\gamma_{i}$ do not meet.
	
	Since the initial jumps are entropic, we only need to show that the increasing parts of the limiting weak solution $u$ between the two curves $\gamma_i$ satisfies the Oleinik condition
	\[
	\partial_xu\leq C.
	\]
	In order to prove this, we use the transport formulation for the derivative at the non-local level, and transfer the property onto the limit.
	
	Let $\delta>0$, with $\delta<\frac{1}{2}\min{\gamma_2-\gamma_1}$, where the minimum is taken over $[0,T]$. It is enough to show that for each such $\delta>0$, the non-local solution $u^{\epsilon}$ satisfies $\partial_{x}u^{\epsilon}(x,t)\leq C$ for $x$ such that $\min{\abs{x-\gamma_{i}(t)}}>\delta$. Without loss of generality, it is enough to show this for values of $x$ lying between the curves $\gamma_i$. Thus, let $\epsilon>0$ be small enough such that $\epsilon<\frac{\delta}{4}$, and
	\[
	\pnorm{L^{\infty}(0,T)}{\gamma_{i}-\gamma_{i}^{\epsilon}}<\dfrac{\delta}{4}
	\]
	for $i=1,2$. Define the domain
	\[
	\mathscr{D}^{\epsilon}_{\delta}=\{(x,t)|t\in[0,T],\gamma_1^{\epsilon}(t)+\epsilon\leq x\leq\gamma_2^{\epsilon}(t)-\epsilon\},
	\]
	which, by our choice of $\epsilon$, contains our set of interest
	\[
	\mathscr{D}_{\delta}=\{(x,t)|t\in[0,T],\gamma_1(t)+\delta\leq x\leq\gamma_2(t)-\delta\}.
	\]
	Since the support of the mollifier $\eta_{\epsilon}$ is $[-\epsilon,\epsilon]$, we have that
	\[
	\partial_x(\eta_{\epsilon}\ast u^{\epsilon})=\eta_{\epsilon}\ast(\partial_xu^{\epsilon})>0
	\]
	on the set $\mathscr{D}^{\epsilon}_{\delta}$. Hence, as long as a characteristic stays in $\mathscr{D}^{\epsilon}_{\delta}$, the spatial derivative along the characteristic is decreasing. Since the initial value of the spatial derivative is bounded above by $C$, we are done if we can show that characteristics do not enter the domain at positive time, since characteristics `fill up' the entire spacetime domain.
	
	Note that the lateral boundaries of $\mathscr{D}^{\epsilon}_{\delta}$ are given by the curves
	\[
	\theta_i^{\epsilon}(t)=\gamma_{i}^{\epsilon}(t)-(-1)^{i}\epsilon
	\]
	for $i=1,2$. Hence, we have that
	\[
	\dot{\theta}_i^{\epsilon}(t)=\dot{\gamma}_{i}^{\epsilon}(t)
	\]
	for $i=1,2$. However,
	\[
	\eta_{\epsilon}\ast u^{\epsilon}(\gamma_1^{\epsilon}(t)+\epsilon,t)<\eta_{\epsilon}\ast u^{\epsilon}(\gamma_1^{\epsilon}(t),t)=\dot{\gamma}_1^{\epsilon}(t)
	\]
	since the jump is entropic, and the term on the left is nothing but the speed of the characteristic passing through the point $(\gamma_1^{\epsilon}(t)+\epsilon)$. Hence, the characteristic cannot be entering the domain $\mathscr{D}^{\epsilon}_{\delta}$ at any positive time. A similar analysis can be carried out for the other boundary curve, which concludes the proof.
	
	Note that for values of $x$ outside the region $\mathscr{D}_{\delta}$, we can carry out a similar analysis looking only at one boundary. Since this is essentially a local result based on analysis of characteristic curves, we can generalise the convergence result to include all initial data that are piece-wise Lipschitz increasing, by extracting further subsequences if necessary and as long as the curves of discontinuity are uniformly apart from each other. Thus we have the following theorem.
	\begin{theorem}\label{pwise lip inc}
		Let $u_0(x)\in BV(\mathbb{R})$ be such that, for real numbers $a_1<a_2<\ldots<a_n$ and uniformly Lipschitz increasing functions $v_0,\ldots,v_n$ with $v_k(a_{k+1})>v_{k+1}(a_{k+1})$:
		\[
		u_0(x)=
		\begin{cases}
			v_0(x),\text{ if }x<a_1, \\
			v_i(x),\text{ if }x\in[a_i,a_{i+1}), 1\leq i\leq n-1, \\
			v_n(x)\text{ if }x\geq a_n.
		\end{cases}
		\]
		Then, there is a positive time $T^{\ast}$ depending on $\{a_k\},\supnorm{u_0}$ such that the non-local solutions $u^{\epsilon}$ of the Cauchy problem \eqref{nonconservative nonlocal} with initial data $u_0$ converge to the entropy solution $u$ of \eqref{burger} with the same initial data.
	\end{theorem}
	\begin{proof}
		Since the characteristics propagate with finite speed, the curves of discontinuity stay apart for some finite time. In particular, we can bound this `secondary catastrophe' time from below by $\frac{1}{2}\supnorm{u_0}D$, where
		\[
		D=\min\{a_k-a_{k-1};k=2,\ldots,n\}.
		\]
		For every convergent subsequence of $u^{\epsilon}$, we can, by extracting further subsequences if necessary, obtain a sequence whose limit satisfies the Oleinik condition by the above argument. Furthermore, we know that the limit must be a weak solution of \eqref{burger}. Hence, in particular every subsequence has a further subsequence converging to the \textit{unique} entropy solution of \eqref{burger} with initial data $u_0$, and thus we conclude that the every sequence of $u^{\epsilon}$ as $\epsilon\to0$ converges to the entropy solution.
	\end{proof}
	Even beyond this `secondary catastrophe' time, the results of Coron et. al. \cite{NLtransport} show that every converging subsequence converges to a \textit{weak} solution, but it is still open as to whether this limit is still the entropy solution, or even whether a unique limit exists at all.
	
	\section{Extension to the isentropic Euler system}\label{isen}
	The isentropic Euler equations with pressure law $p(\rho)=\rho^3/3$ allows for a natural extension of our inviscid regularisation schema due to the structure of the associated Riemann invariants and eigenvalues. The equations of gas dynamics for a pressure law as above can be written in conservative form as follows:
	\begin{equation}\label{isentropic euler}
		\begin{split}
			\partial_t\rho+\partial_x(\rho v)&=0, \\
			\partial_t(\rho v)+\partial_x\left(\rho v^2+\dfrac{\rho^3}{3}\right)&=0.
		\end{split}
	\end{equation}
	The eigenvalues of this system are given by $v\pm\rho$, and the Riemann invariants are given by $\rho\pm v$, and hence for classical solutions we have that
	\begin{equation}\label{riemann invariants}
		\begin{split}
			(\rho+v)+(\rho+v)\partial_x(\rho+v)&=0, \\
			(\rho-v)-(\rho+v)\partial_x(\rho-v)&=0.
		\end{split}
	\end{equation}
	Note that this is similar in form to Burgers equation in non-conservative form. Thus, let $\mu=\rho+u,\lambda=\rho-\mu$, and consider the analogous inviscid regularisation of these equations:
	\begin{equation}\label{reg riemann}
		\begin{split}
			\partial_t\mu^{\epsilon}+(\eta_{\epsilon}\ast\mu^{\epsilon})\partial_x\mu^{\epsilon}&=0, \\
			\partial_t\lambda^{\epsilon}-(\eta_{\epsilon}\ast\lambda^{\epsilon})\partial_x\lambda^{\epsilon}&=0.
		\end{split}
	\end{equation}
	Since we have decomposed the system of conservation laws into a pair of decoupled scalar equations, we can extend Theorem~\ref{smoothconv} to the isentropic system \eqref{isentropic euler}. Note that as we are assuming a classical regime, we can equivalently work with Riemann invariants instead of the conserved quantities.
	
	\section{Conclusion}\label{conclusion}
	In contrast to the standard conservative method for non-local regularisation of scalar conservation laws, we have demonstrated here the merits of an explicitly non-conservative approach, showcasing the strong maximum principle and total variation-preserving properties of \eqref{nonconservative nonlocal}. In the smooth regime, and therefore locally-in-time for any Lipschitz initial data, our a priori estimates help us justify the formal limit of the non-local equations as $\epsilon\to0^+$. In a sense, this result is also sharp, since for even the simplest cases of discontinuous initial data, convergence to the entropy solution does not hold in general, except in particular for regularisations of the Burgers equation.
	
	One interesting question that we were unfortunately not able to resolve was that of the singular limit for Lipschitz initial data beyond the catastrophe time. Since non-entropic shocks cannot form at positive times, the Riemann-type counterexamples do not pose an obstacle to convergence. However, we were not able to justify the positive result. Of course, by total variation preservation, the sequence $u^{\epsilon}$ has an $L^1$-convergent subsequence on compact subsets of $\Omega_T$ for every $T$, but it is not clear whether the limit is always the appropriate entropy solution.
	
	It is easy to see that for any initial data with a rarefaction-type discontinuity, the nonlocal-to-local limit does not hold. However, this failure of convergence is not so obvious if the data does not have any positive jumps. Adapting a counter-example from \cite{counters}, we can show that our schema \eqref{nonconservative nonlocal} satisfies the nonlocal-to-local limit in some cases where the conservative regularisation \eqref{conservative nonlocal} fails to do so. In particular, consider initial data $u_0$ such that
	\[
	u_0(x)=
	\begin{cases}
		0&\text{ if }\abs{x}\geq2, \\
		-\operatorname{sgn}(x)&\text{ if }\abs{x}\leq1, \\
	\end{cases}
	\]
	and Lipschitz monotone non-decreasing on $(-\infty,0)\cup(0,\infty)$. Note that such a function $u_0$ cannot be quasiconcave either. For the conservative regularisation \eqref{conservative nonlocal}, the nonlocal-to-local limit does not hold, even weakly, for any time interval. However, by construction, $u_0$ is a piece-wise Lipschitz increasing function with entropic discontinuity. Hence, by Theorem~\ref{pwise lip inc} we have convergence to the entropy solution in the singular limit with respect to the non-conservative non-local regularisation \eqref{nonconservative nonlocal}. Since there is only one discontinuity, this convergence is in fact global in time.
	
	It would be interesting to demonstrate that the nonlocal-to-local limit \textit{always} holds, provided the initial data does not contain any positive jumps, or construct a counterexample to convergence. However, as we can see the standard counterexamples for symmetric convolutional kernels from the theory of conservative nonlocal equations do not carry over. Furthermore, even in the presence of positive jumps, one could have convergence to a \textit{unique} non-entropic weak solution, as the Riemann data shows. A precise characterisation of the singular limits would be quite valuable.
	
	\section{Acknowledgements}
	The first two authors would like to thank the Department of Atomic Energy, Government of India, for their support under project no. 12-R\&D-TFR-5.01-0520. E.W. acknowledges funding by the Deutsche Forschungsgemeinschaft (DFG,
	German Research Foundation) within SPP 2410, project number 525716336.
	
	\bibliographystyle{plain}
	\bibliography{citations}

\begin{thebibliography}{10}

\bibitem{opform}
G.~Aletti, G.~Naldi, and G.~Toscani.
\newblock First‐order continuous models of opinion formation.
\newblock {\em SIAM Journal on Applied Mathematics}, 67(3):837--853, 2007.

\bibitem{boundary}
A.~Bayen, J.-M. Coron, N.~De~Nitti, A.~Keimer, and L.~Pflug.
\newblock Boundary controllability and asymptotic stabilization of a nonlocal
  traffic flow model.
\newblock {\em Vietnam Journal of Mathematics}, 49(3):957--985, Sep 2021.

\bibitem{multilane}
A.~Bayen, J.~Friedrich, A.~Keimer, L.~Pflug, and T.~Veeravalli.
\newblock Modeling multilane traffic with moving obstacles by nonlocal balance
  laws.
\newblock {\em SIAM Journal on Applied Dynamical Systems}, 21(2):1495--1538,
  2022.

\bibitem{sed}
F.~Betancourt, R.~Bürger, K.~H. Karlsen, and E.~M. Tory.
\newblock On nonlocal conservation laws modelling sedimentation.
\newblock {\em Nonlinearity}, 24(3):855, feb 2011.

\bibitem{ccdnkp}
G.~M. Coclite, J.-M. Coron, N.~De~Nitti, A.~Keimer, and L.~Pflug.
\newblock A general result on the approximation of local conservation laws by
  nonlocal conservation laws: The singular limit problem for exponential
  kernels.
\newblock {\em Annales de l’Institut Henri Poincaré C, Analyse non
  linéaire}, 40(5):1205–1223, November 2022.

\bibitem{bvkernel}
G.~M. Coclite, N.~De~Nitti, A.~Keimer, and L.~Pflug.
\newblock On existence and uniqueness of weak solutions to nonlocal
  conservation laws with {BV} kernels.
\newblock {\em Zeitschrift f{\"u}r angewandte Mathematik und Physik},
  73(6):241, Oct 2022.

\bibitem{anisotropic}
M.~Colombo, G.~Crippa, E.~Marconi, and L.~V. Spinolo.
\newblock Local limit of nonlocal traffic models: Convergence results and total
  variation blow-up.
\newblock {\em Annales de l'Institut Henri Poincaré C, Analyse non linéaire},
  38(5):1653--1666, 2021.

\bibitem{tvblowup}
M.~Colombo, G.~Crippa, E.~Marconi, and L.~V. Spinolo.
\newblock Local limit of nonlocal traffic models: Convergence results and total
  variation blow-up.
\newblock {\em Annales de l'Institut Henri Poincaré C, Analyse non linéaire},
  38(5):1653--1666, 2021.

\bibitem{counters}
M.~Colombo, G.~Crippa, and L.~V. Spinolo.
\newblock On the singular local limit for conservation laws with nonlocal
  fluxes.
\newblock {\em Archive for Rational Mechanics and Analysis}, 233(3):1131--1167,
  Sep 2019.

\bibitem{control}
J.-M. Coron and Z.~Wang.
\newblock Controllability for a scalar conservation law with nonlocal velocity.
\newblock {\em Journal of Differential Equations}, 252(1):181--201, 2012.

\bibitem{NLtransport}
Jean-Michel Coron, Alexander Keimer, and Lukas Pflug.
\newblock Nonlocal transport equations---existence and uniqueness of solutions
  and relation to the corresponding conservation laws.
\newblock {\em SIAM Journal on Mathematical Analysis}, 52:5500--5532, 10 2020.

\bibitem{Crippa2013}
G.~Crippa and M.~L{\'e}cureux-Mercier.
\newblock Existence and uniqueness of measure solutions for a system of
  continuity equations with non-local flow.
\newblock {\em Nonlinear Differential Equations and Applications NoDEA},
  20(3):523--537, Jun 2013.

\bibitem{NLsource}
Camillo De~Lellis, Piotr Gwiazda, and Agnieszka {\'{S}}wierczewska-Gwiazda.
\newblock Transport equations with integral terms: existence, uniqueness and
  stability.
\newblock {\em Calculus of Variations and Partial Differential Equations},
  55(5):128, Oct 2016.

\bibitem{transport}
R.~J. DiPerna and P.-L. Lions.
\newblock Ordinary differential equations, transport theory and sobolev spaces.
\newblock {\em Inventiones mathematicae}, 98(3):511--547, Oct 1989.

\bibitem{gottlich}
J.~Friedrich, S.~G{\"o}ttlich, A.~Keimer, and L.~Pflug.
\newblock Conservation laws with nonlocal velocity: The singular limit problem.
\newblock {\em SIAM J. Appl. Math.}, 84:497--522, 2022.

\bibitem{lwrnonlocal}
J.~Friedrich, O.~Kolb, and S.~Göttlich.
\newblock A $\text{G}$odunov type scheme for a class of $\text{LWR}$ traffic
  flow models with non-local flux.
\newblock {\em Networks and Heterogeneous Media}, 13(4):531--547.

\bibitem{lwrnonlocalwp}
P.~Goatin and S.~Scialanga.
\newblock Well-posedness and finite volume approximations of the $\text{LWR}$
  traffic flow model with non-local velocity.
\newblock {\em Networks and Heterogeneous Media}, 11(1):107--121.

\bibitem{Hormander}
L.~Hormander.
\newblock {\em Lectures on nonlinear hyperbolic differential equations},
  page~1.
\newblock Math{\'e}matiques et Applications. Springer, Berlin, Germany, 1997
  edition, July 1997.

\bibitem{supply}
A.~Keimer, G.~Leugering, and T.~Sarkar.
\newblock Analysis of a system of nonlocal balance laws with weighted work in
  progress.
\newblock {\em Journal of Hyperbolic Differential Equations}, 2018.

\bibitem{survey}
A.~Keimer and L.~Pflug.
\newblock Existence, uniqueness and regularity results on nonlocal balance
  laws.
\newblock {\em Journal of Differential Equations}, 263(7):4023--4069, 2017.

\bibitem{monodatanonlocal}
A.~Keimer and L.~Pflug.
\newblock On approximation of local conservation laws by nonlocal conservation
  laws.
\newblock {\em Journal of Mathematical Analysis and Applications},
  475(2):1927--1955, 2019.

\bibitem{kruzkov}
S.~N. Kružkov.
\newblock First order quasilinear equations in several independent variables.
\newblock {\em Mathematics of the USSR-Sbornik}, 10(2):217, feb 1970.

\bibitem{Norgard_2008}
Greg Norgard and Kamran Mohseni.
\newblock A regularization of the burgers equation using a filtered convective
  velocity.
\newblock {\em Journal of Physics A: Mathematical and Theoretical},
  41(34):344016, aug 2008.

\bibitem{ConFil}
Greg Norgard and Kamran Mohseni.
\newblock On the convergence of the convectively filtered burgers equation to
  the entropy solution of the inviscid burgers equation.
\newblock 2009.

\bibitem{timecompactness}
J.~Simon.
\newblock Compact sets in the space ${L}^p(0,{T};{B})$.
\newblock {\em Annali di Matematica Pura ed Applicata}, 146:65--96, 1986.

\bibitem{Zum}
K.~Zumbrun.
\newblock On a nonlocal dispersive equation modeling particle suspensions.
\newblock {\em Quarterly of Applied Mathematics}, 57, 09 1999.

\end{thebibliography}
\end{document}